\documentclass[12pt]{amsart}

\usepackage{amsmath,amssymb,mathrsfs}

\usepackage{amsthm} 
\usepackage{verbatim}
\usepackage{srcltx} 
\usepackage{times}
\usepackage{amsmath,amssymb,mathrsfs}

\usepackage{amsthm} 
\usepackage{srcltx}

 \usepackage[dvips]{graphicx,color}

\usepackage{amsopn,cite,mathrsfs}
\usepackage{amsfonts}
\usepackage{eucal}

\numberwithin{equation}{section}

\newtheorem{thm}{Theorem}[section]
\newtheorem{lemma}[thm]{Lemma}

\newtheorem*{proof-mthm}{Proof of Main Theorem}

\newcommand{\C}{{\mathbb C}}

 \newcommand{\blem}{\begin{lemma}}
 \newcommand{\elem}{\end{lemma}}
\newcommand{\bpve}{\begin{proof}}
\newcommand{\epve}{\end{proof}}

\def\C{\mathbb{C}}

\newcommand{\BCn}{\mathbb B_n}
\newcommand{\Bn}{\mathbb B^n}
\newcommand{\Rn}{\mathbb R^n}
\newcommand{\Dh}{\Delta_h}

\usepackage{amssymb}
\usepackage{amsmath}
\usepackage{amsfonts}
\usepackage[colorlinks=true,linkcolor=red]{hyperref}
\newtheorem{theorem}{Theorem}
\theoremstyle{plain}

\numberwithin{equation}{section}

\numberwithin{theorem}{section}

\numberwithin{proposition}{section}

\title{A Converse Mean-Value Property}
\author[M.~Moravík]{Matěj Moravík}
\address{Mathematics Institute, Silesian University in Opava,
 Na~Rybn\'\i\v cku~ 626/1, 74601~Opava, Czech Republic}
 \email{Mor0109@slu.cz}
 
\author[E.H.~Youssfi]{El Hassan Youssfi}
\address{Aix-Marseille Universit\'e,
Institut de Math\'ematiques de Marseille (I2M),  UMR 7373,
Site de Saint Charles, 3~place Victor Hugo, Case~19,
13331~Marseille C\'edex~3, France}
\email{el-hassan.youssfi{@}univ-amu.fr}
\thanks{Research of M.~Moravik supported by  GA CR grant no. 25-18042S.
}

\begin{document}

\maketitle

\begin{abstract}

We prove a converse mean-value theorem for harmonic functions associated with the invariant Laplace--Beltrami operator on the real unit ball. Under suitable integrability and topological assumptions on the domain and its Green potential, we show that the invariant volume mean-value property characterizes hyperbolic balls centered at the origin. As a consequence, we also answer a question left open by Bruna and D\'etraz in the setting of invariantly harmonic functions.
\end{abstract}

\section{Introduction and statement of the main result}

The mean-value property plays a central role in potential theory and harmonic analysis.
For classical harmonic functions in Euclidean space, converse mean-value theorems
characterize balls among large classes of domains see \cite{Ep}, \cite{Ku} and \cite{AG}. In the complex unit ball,
Bruna and D\'etraz \cite{BD} proved that Bergman balls are the only relatively compact
connected domains enjoying the invariant volume mean-value property for invariantly
harmonic functions. Their proof is based on Green potentials, tangential vector
fields and a radial monotonicity argument. 

The purpose of the present paper is to establish a real-hyperbolic converse mean-value theorem for
the invariant Laplace--Beltrami operator which characterizes balls among large classes of domains satisfying some integrability conditions, which include, in particular, the relatively compact case. 
Our method can be adapted to solve an open question stated in \cite{BD}.

The aim of this note is to formulate and prove a converse mean-value theorem
for the real hyperbolic invariant Laplacian on the the unit ball  $
\Bn=\{x\in\Rn:\ |x|<1\}$ in ${\mathbb R}^n, n\geq 2$.  More precisely, for $f \in C^{2}(\Bn)$, the Laplace-Beltrami operator associated with the hyperbolic metric is given by (up to a constant factor)

$$
\Dh f(x)=\left(1-|x|^{2}\right)\left[\left(1-|x|^{2}\right) \Delta f(x)+2(n-2)\langle x, \nabla f(x)\rangle\right],
$$
where $\Delta f=\frac{\partial^{2} f}{\partial x_{1}^{2}}+\cdots+\frac{\partial^{2} f}{\partial x_{n}^{2}}$ and $\nabla f=\left(\frac{\partial f}{\partial x_{1}}, \ldots, \frac{\partial f}{\partial x_{n}}\right)$ are the Euclidean Laplacian and gradient. It  satisfies  $\Delta_{h} f(a)=\Delta\left(f \circ \varphi_{a}\right)(0)$, where

$$
\varphi_{a}(x)=\frac{a|x-a|^{2}+\left(1-|a|^{2}\right)(a-x)}{[x, a]^{2}},
$$
is the Moebius transformation mapping $\Bn$ to $\Bn$ and exchanging $a$ and 0 . Here $[x, a]$ is defined as

$$
[x, a]:=\sqrt{1-2\langle x, a\rangle+|x|^{2}|a|^{2}} \quad(x, a \in \Bn) .
$$

Let $\Omega$ be an open set in $\Bn$.  A complex-valued function $f \in C^{2}(\Omega)$ is called hyperbolic harmonic or $H$-harmonic  on $\Omega$, if $\Delta_{h} f(x)=0$ for all $x \in \Omega$.

Throughout, let 
$
\Bn=\{x\in\Rn:\ |x|<1\}$ be the unit ball in ${\mathbb R}^n$ 
and 
\[d\nu(x)=\frac{dx}{(1-|x|^2)^n}.
\]
denote the measure is invariant under the Moebius group.

 Our main result is the follwing:
\begin{theorem}\label{main}
Suppose that \(n\ge2\) and let \(\Omega\subset \Bn\) be an open connected set  containing the origin $0$ such that
\begin{eqnarray}\label{finiteVol}
\nu(\Omega)<+\infty.
\end{eqnarray}
Assume that for every compact set \(K\Subset\mathbb B^n\),
\begin{eqnarray}\label{C1cond}
\sup_{z\in K}
\int_\Omega |\xi-z|^{1-n}\,d\nu(\xi)
<+\infty,
\end{eqnarray}
\begin{eqnarray}\label{C1bcond}
\lim_{r\to1^-} (1-r^2)^{n-2}
\sup_{z\in\Omega\cap\{|z|=r\}}
\int_\Omega |\xi-z|^{1-n}\,d\nu(\xi)=0
\end{eqnarray}
and
\begin{eqnarray}\label{Topcond}
\partial\Omega\cap\mathbb B^n =\partial\overline{\Omega}\cap\mathbb B^n.
\end{eqnarray}
%%%%%%%%
%%%%%%
%%%%%%%%%%%%
%%%%%%%%%
Suppose that, for every $\nu$-integrable  \(H\)-harmonic function \(f\) in a neighbourhood of
\(\overline{\Omega} \cap\mathbb B^n\), one has
\[
\frac1{\nu(\Omega)}\int_\Omega f(x)\,d\nu(x)=f(0).
\]
Then \(\Omega\) is a hyperbolic ball centered at \(0\). Equivalently, there
exists \(0<R<1\) such that
\[
\Omega=\{x\in\Bn:\ |x|<R\}.
\]
\end{theorem}

\section{Preparatory results}

Let \(G(\xi,z)\) be the Green function associated with \(\Dh\). We use the
sign convention
\[
\Delta_{h,z} G(\xi,z)=\delta_\xi(z)
\]
in the sense of distributions. With this convention, the Green potential
\[
([G]\Psi)(z):=\int_{\Bn}\Psi(\xi)G(\xi,z)\,d\nu(\xi)
\]
satisfies
\[
\Dh ([ G]\Psi)=\Psi,
\]
for all $C^\infty$-compactly supported functions $\Psi$ in $\Bn.$
The Green function is symmetric:
\[
G(\xi,z)=G(z,\xi),
\]
and, for fixed \(z \in \mathbb B^n\), the function
\[
\xi\longmapsto G(\xi,z)
\]
is \(H\)-harmonic away from \(\xi=z\).
In particular, if \(z\notin\overline{\Omega}\), then
\[
\xi\longmapsto G(\xi,z)
\]
is a $\nu$-integrable \(H\)-harmonic in a neighbourhood of \(\overline{\Omega}\).
Put
\[
c=\nu(\Omega).
\]
The mean-value hypothesis says
\[
\int_\Omega f(\xi)\,d\nu(\xi)=c f(0)
\]
for every $\nu$-integrable \(H\)-harmonic function \(f\) in a neighborhoord of  \(\overline{\Omega}\cap\mathbb B^n\).
Taking
\[
f(\xi)=G(\xi,z),
\qquad z\notin\overline{\Omega},
\]
we obtain
\[
\int_\Omega G(\xi,z)\,d\nu(\xi)=cG(0,z),
\qquad z\notin\overline{\Omega}.
\]
Define
\begin{eqnarray}\label{h}
h(z):=\int_\Omega G(\xi,z)\,d\nu(\xi).
\end{eqnarray}
Then
\begin{eqnarray}\label{hG}
h(z)=cG(0,z),
\qquad z\notin\overline{\Omega}.
\end{eqnarray}

%%%%%%%%%%%%%
For 
 \(1\le i<j\le n\),
let
\[
T_{ij}
=
x_i\frac{\partial}{\partial x_j}
-
x_j\frac{\partial}{\partial x_i}.
\]
%%%%%%%
\begin{lemma}\label{gradient}
 Then there exists a constant \(C>0\) such that
\[
|\nabla_zG(\xi,z)|
\le
C(1-|z|^2)^{n-2}|\xi-z|^{1-n},
\qquad
\xi,z\in\Bn,\quad \xi\ne z.
\]
Consequently,
\[
|T_{ij,z}G(\xi,z)|
\le
C(1-|z|^2)^{n-2}|\xi-z|^{1-n},
\qquad
\xi,z\in\Bn,\quad \xi\ne z.
\]
\end{lemma}
\begin{proof} It is known \cite{St} that the Green function associated with the invariant Laplacian may be written in
the form
$$G(\xi, z) = \mathcal G(\rho(\xi,z))$$
where 
 \[
\mathcal G(r)
=
-C_n \int_{s}^1\frac{(1-r)^{\,n-2}}{r^{n/2}} dr
\]
for a positive constant $C_n$, and
\[
\rho(\xi,z)
=
|\varphi_z(\xi)|^2
=
\frac{|\xi-z|^2}
     {1-2\langle \xi,z\rangle+|\xi|^2|z|^2}.
\]
Let
\[
B(\xi,z)
:=
1-2\langle \xi,z\rangle+|\xi|^2|z|^2.
\]
%%%%%%%%%%%%%%%%
We will use the identities
\[
B(\xi,z)=|\xi-z|^2+(1-|\xi|^2)(1-|z|^2)
\]
and
\[
1-\rho(\xi,z)
=
\frac{(1-|\xi|^2)(1-|z|^2)}{B(\xi,z)}.
\]
Since
\[
|\mathcal G'(s)|
\le
C s^{-n/2}(1-s)^{n-2},
\qquad 0<s<1.
\]
Therefore
\[
|\mathcal G'(\rho(\xi,z))|
\le
C
|\xi-z|^{-n}
B(\xi,z)^{2-\frac n2}
(1-|\xi|^2)^{n-2}
(1-|z|^2)^{n-2}.
\]
We now estimate \(\nabla_z\rho\). Differentiating the formula for \(\rho\)
gives
\[
\nabla_z\rho(\xi,z)
=
\frac{2}{B(\xi,z)^2}
\left[
B(\xi,z)(z-\xi)
+
|\xi-z|^2(\xi-|\xi|^2z)
\right].
\]
Moreover,
\[
|\xi-|\xi|^2z|^2
=
|\xi|^2B(\xi,z)
\le
B(\xi,z).
\]
Since \(B(\xi,z)\ge |\xi-z|^2\), it follows that
\[
|\nabla_z\rho(\xi,z)|
\le
C\frac{|\xi-z|}{B(\xi,z)}.
\]
By the chain rule,
\[
\nabla_zG(\xi,z)
=
\mathcal G'(\rho(\xi,z))\,\nabla_z\rho(\xi,z).
\]
Hence
\[
|\nabla_zG(\xi,z)|
\le
C
(1-|z|^2)^{n-2}
|\xi-z|^{1-n}
\frac{(1-|\xi|^2)^{n-2}}{B(\xi,z)^{\frac{n-2}{2}}}.
\]
Finally,
\[
B(\xi,z)\ge (1-|\xi||z|)^2
\ge (1-|\xi|)^2.
\]
Since \(1-|\xi|^2\le 2(1-|\xi|)\), we obtain
\[
\frac{(1-|\xi|^2)^{n-2}}{B(\xi,z)^{\frac{n-2}{2}}}
\le C.
\]
Therefore
\[
|\nabla_zG(\xi,z)|
\le
C(1-|z|^2)^{n-2}|\xi-z|^{1-n}.
\]
It remains to prove the tangential estimate. Since
\[
T_{ij,z}
=
z_i\frac{\partial}{\partial z_j}
-
z_j\frac{\partial}{\partial z_i},
\]
we have
\[
|T_{ij,z}G(\xi,z)|
\le
|z_i|\left|\frac{\partial G}{\partial z_j}(\xi,z)\right|
+
|z_j|\left|\frac{\partial G}{\partial z_i}(\xi,z)\right|.
\]
Since \(|z_i|,|z_j|\le |z|<1\), this gives
\[
|T_{ij,z}G(\xi,z)|
\le
2|\nabla_zG(\xi,z)|.
\]
The desired estimate for \(T_{ij,z}G\) follows immediately.
\end{proof}

\begin{lemma}\label{Regularity} Under (\ref{finiteVol}) and (\ref{C1cond}), the function $h$ given by (\ref{h})
satisfies
\begin{enumerate}
\item \(h\in C^1(\Bn)\);
\item \(\Delta_h h=\chi_\Omega\) in \(\mathcal D'(\Bn)\);
\item \(\Delta_h h=1\) classically in \(\Omega\).
\end{enumerate}
\end{lemma}

\begin{proof} 
Fix a compact set \(K\Subset\Bn\). Lemma \ref{gradient} implies that
\[
|\nabla_zG(\xi,z)|
\le C_K |\xi-z|^{1-n},
\qquad z\in K.
\]
By assumption,
\[
\sup_{z\in K}
\int_\Omega |\xi-z|^{1-n}\,d\nu(\xi)<+\infty.
\]
Hence
\[
\sup_{z\in K}
\int_\Omega |\nabla_zG(\xi,z)|\,d\nu(\xi)<+\infty.
\]
Therefore differentiation under the integral sign is justified and
\[
\partial_{z_j}h(z)
=
\int_\Omega \partial_{z_j}G(\xi,z)\,d\nu(\xi).
\]

The same domination allows the use of dominated convergence, proving that the
first derivatives are continuous. Thus
\[
h\in C^1(\Bn).
\]
Let \(\varphi\in C_c^\infty(\mathbb B^n)\). Using Fubini's theorem,
\[
\langle \Delta_h h,\varphi\rangle
=
\langle h,\Delta_h\varphi\rangle
=
\int_\Omega
\left(
\int_{\Bn}
G(\xi,z)\Delta_h\varphi(z)\,d\nu(z)
\right)
d\nu(\xi).
\]
Since
\[
\Delta_{h,z}G(\xi,z)=\delta_\xi(z),
\]
we obtain
\[
\int_{\Bn}
G(\xi,z)\Delta_h\varphi(z)\,d\nu(z)
=
\varphi(\xi).
\]
Hence
\[
\langle \Delta_h h,\varphi\rangle
=
\int_\Omega \varphi(\xi)\,d\nu(\xi)
=
\langle \chi_\Omega,\varphi\rangle.
\]
Therefore
\[
\Delta_h h=\chi_\Omega
\]
in the sense of distributions.

Inside \(\Omega\), the right-hand side equals the smooth function \(1\).
Elliptic regularity for \(\Delta_h\) implies that \(h\) is \(C^\infty\) in
\(\Omega\), and consequently
\[
\Delta_h h=1
\]
holds pointwise in \(\Omega\).
\end{proof}
\begin{lemma}\label{Breg}
Under (\ref{finiteVol}), (\ref{C1cond}) and (\ref{C1bcond}) , the function $h$ given by (\ref{h})
satisfies
\begin{eqnarray}\label{hbreg}
\lim_{r\to1^-}
\sup_{z\in\Omega\cap\{|z|=r\}}
|T_{ij}h(z)|
=0
\end{eqnarray}
for all 
 \(1\le i<j\le n\).
\end{lemma}
\begin{proof} Let  \(1\le i<j\le n\).
By  Lemma \ref{gradient} and  \ref{Regularity} we know that  \(h\in C^1(\mathbb B^n)\) and we have
\[
T_{ij}h(z)
=
\int_\Omega T_{ij,z}G(\xi,z)\,d\nu(\xi).
\]
By Lemma \ref{gradient} there is a constant $C>0$ such  that
\[
|T_{ij,z}G(\xi,z)|
\le
C(1-|z|^2)^{n-2}|\xi-z|^{1-n}.
\]
uniformly in $z, \in  \mathbb B^n$ such that $z\not =\xi.$ Consequently,
\[
|T_{ij}h(z)|
\le
\int_\Omega |T_{ij,z}G(\xi,z)|\,d\nu(\xi)
\le
C(1-|z|^2)^{n-2}\left(\int_\Omega |\xi-z|^{1-n}\,d\nu(\xi)\right).
\]
Therefore, 
\[
\sup_{z\in\Omega\cap\{|z|=r\}}
|T_{ij}h(z)|
\le
C
\sup_{z\in\Omega\cap\{|z|=\rho\}}
\int_\Omega |\xi-z|^{1-n}\,d\nu(\xi)
\]
so that by the assumption (\ref{C1bcond}) it follows that
\[
\lim_{r\to1^-}
\sup_{z\in\Omega\cap\{|z|=r\}}
|T_{ij}h(z)|
=0.
\]
This proves the lemma.
\end{proof}

%%%%%%%%%%%
%%%%%%
The following lemma is the content of Proposition 19 and the text before in \cite{BEY}. 
\begin{lemma}\label{g}
For  $0\leq t<1,$ let 
\[
g_0(t)
=
-\frac{1}{4(n-1)}\log(1-t)
+
\frac{n-2}{4n(n-1)}\,t\,
{}_3F_2\left(
2-\frac n2,1,1;
1+\frac n2,2;
t
\right),
\]
here ${}_3F_2$ denotes the generalized hypergeometric function. For $x\in \mathbb B, $ 
define
$
g(x)=g_0(|x|^2).
$
Then we have $\Dh g=1$. 
\end{lemma}

%%%%%%%%%%

\section{Proof of the main result}

We now prove Theorem \ref{main}. Recall
\[
c=\nu(\Omega).
\]
By the mean-value hypothesis applied to the functions
\[
\xi\longmapsto G(\xi,z),
\qquad z\notin \overline{\Omega}\cap\mathbb B^n,
\]
we have already obtained
\[
h(z)=\int_\Omega G(\xi,z)\,d\nu(\xi)=cG(0,z),
\qquad z\notin \overline{\Omega}\cap\mathbb B^n.
\]
By Lemma \ref{Regularity},
\[
h\in C^1(\Bn)
\]
and
\[
\Delta_h h=1
\qquad\text{in }\Omega.
\]
By Lemma \ref{g} we also have
\[
\Delta_h g=1.
\]
Define
\[
u(z)=h(z)-g(z),
\]
then
\[
\Delta_h u=0
\qquad\text{in }\Omega.
\]
Thus \(u\) is \(H\)-harmonic in \(\Omega\).\\
By (\ref{hG}),
\[
u(z)=cG(0,z)-g(z),
\qquad z\in\mathbb B^n\setminus\overline\Omega.
\]
The right-hand side is radial, since both \(G(0,z)\) and \(g(z)\) depend only
on \(|z|\). Hence there exists a one-variable function \(F\) such that
\[
u(z)=F(|z|)
\qquad\text{for }z\in\mathbb B^n\setminus\overline\Omega.
\]

Because \(h\in C^1(\mathbb B^n)\) and \(g\in C^\infty(\Bn)\), the
identity above may be differentiated in every direction. In particular, if \(T\) is tangent to the Euclidean sphere
\(\{|z|=\mathrm{constant}\}\), then
\[
Tu=0
\qquad\text{on }\,\mathbb B^n\setminus\overline\Omega.
\]

We recall that for \(1\le i<j\le n\), the vector fields
\[
T_{ij}
=
x_i\frac{\partial}{\partial x_j}
-
x_j\frac{\partial}{\partial x_i}.
\]
are the infinitesimal generators of rotations. They satisfy
\[
T_{ij}(|x|^2)=0,
\]
and therefore are tangent to all Euclidean spheres centered at \(0\).

Since \(\Delta_h\) is invariant under the orthogonal group \(O(n)\), it
commutes with rotations. Hence
\[
[\Delta_h,T_{ij}]=0.
\]
Consequently,
\[
\Delta_h(T_{ij}u)
=
T_{ij}(\Delta_h u)
=
0
\qquad\text{in }\Omega.
\]
Thus \(T_{ij}u\) is \(H\)-harmonic in \(\Omega\).
Moreover, from the previous step,
\begin{eqnarray}\label{Tiju} 
T_{ij}u=0
\qquad\text{on }\,\mathbb B^n\setminus\overline\Omega.
\end{eqnarray}
By continuity of $u$ on $\mathbb B^n$ and  the assumption (\ref{Topcond}) we see that  the equality $(\ref{Tiju} )$ holds $\partial \Omega \cap \mathbb B^n.$  

Next consider the the
exhaustions
\[
\Omega_r=\Omega\cap\{|x|<r\},
\qquad 0<r<1.
\]
The boundary of \(\Omega_r\) consists of the finite boundary part
\(\partial\Omega\cap\{|x|<r\}\) and the  trace on $\Omega$ of the sphere with radius $r$,
\(\Omega\cap\{|x|=r\}\). 
Using $(\ref{Tiju} )$ and applying  the maximum principle (see~\cite{St}, 4.1.6]) on these
exhaustions we obtain
\begin{eqnarray*}
[T_{ij}u(z)] \leq 
\sup_{z\in\Omega\cap\{|z|=r\}}
|T_{ij}h(z)|
\end{eqnarray*}
 for all $z \in \Omega_r.$ Applying  Lemma \ref{Breg} and  taking the limit as $r\to 1$ yeilds

\begin{eqnarray}\label{Tijuglobal} T_{ij}u=0
\qquad\text{in }\Omega.
\end{eqnarray}
Hence
\[
T_{ij}u=0
\quad\text{on $\mathbb B^n$,}\quad\forall\,1\le i<j\le n.
\]
Since the vector fields \(T_{ij}\) span the tangent space to each Euclidean
sphere \(\{|x|=r\}\), we conclude that \(u\) is constant on any such sphere. Thus u is radial on $\mathbb B^n$.
Since \(0\in\Omega\) and \(\Omega\) is open, there exists \(r_0>0\) such the open ball  $B(0,r_0)$ centered at $0$ with radius  \(r_0\) 
satisfies 
$
B(0,r_0)\subset \Omega.$
On this ball, \(u\) is radial and  \(H\)-harmonic, therefore it is constant (see Corollary 4.1.3 in \cite{St}). By unique continuation for elliptic equations, \(u\) is constant in the whole
domain \(\Omega\).
 Thus there exists a constant \(k\in\mathbb R\)
such that
\[
u\equiv k
\qquad\text{in }\Omega.
\]
Since \(u=h-g\) and \(u\equiv k\), the boundary identity becomes
\[
cG(0,z)-g(z)=k,
\qquad z\in\partial\Omega\cap\Bn.
\]
The left-hand side is a radial function. Thus, setting \(r=|z|\), there is a
function
\[
F(r)=cG(0,r)-g(r)
\]
such that
\[
F(|z|)=k,
\qquad z\in\partial\Omega\cap\Bn.
\]
 The Green function
\(G(0,r)\) has a logarithmic-type singularity at \(r=0\), for $n = 2$, while \(G(0,r) \asymp r^{1-\frac{n}{2}}\) at $r =0$ for $n>2.$ The function
\(g(r)\) has the boundary behaviour
\[
g(r)\to+\infty
\qquad\text{as }r\to1^-.
\]
Moreover, differentiating the explicit radial formulas shows that
\begin{align*}
F'(r)&= c\frac{(1-r)^{n-2}}{r^{\frac{n}{2}}} - \frac{r}{2(n-1)(1-r^2)}
+
\frac{(n-2)r}{2n(n-1)}\\
&\quad+\,{}_2F_1\left(
2-\frac n2,1;
1+\frac n2;
r^2
\right)
\end{align*}
and changes sign at most once. Hence \(F\) increases from
\(-\infty\) to a maximum value and then decreases again to \(-\infty\) as
\(r\to1^-\).

Therefore the equation
\[
F(r)=k
\]
has at most two solutions in \(0<r<1\). Consequently \(|z|\) takes at most two
values on
\[
\partial\Omega\cap\mathbb B^n.
\]

Thus there exist radii \(0<r_1\le r_2<1\) such that
\[
\partial\Omega\cap\mathbb B^n
\subset
S_{r_1}\cup S_{r_2},
\qquad
S_r=\{x\in\mathbb R^n:\ |x|=r\}.
\]

Since \(\Omega\) is connected and contains \(0\), the two-radius alternative is
impossible. Indeed, if two distinct radii occurred, then the only connected
region lying between the two corresponding spheres would be the annulus
\[
r_1<|x|<r_2,
\]
which does not contain \(0\). The other possible configuration, namely the
union of an inner ball and an exterior region, is disconnected. Hence only one
radius occurs.
Therefore, for some \(R\in(0,1)\),
\[
\partial\Omega\cap\Bn=S_R.
\]
Using the topological assumption
\[
\partial\Omega\cap\Bn
=
\partial\overline{\Omega}\cap\Bn,
\]
we exclude hidden boundary components. Since \(\Omega\) is connected and
contains \(0\), it follows that
\[
\Omega=\{x\in\Bn:\ |x|<R\}.
\]
Thus \(\Omega\) is a Euclidean ball centered at \(0\). Euclidean balls centered
at \(0\) are exactly hyperbolic balls centered at \(0\) in the ball model.
This proves the theorem.

\section{Concluding remarks} 1)  As mentioned in the introduction, our method can be adapted to solve the question raised by Bruna-D\'etraz \cite{BD} concerning the extension of the their result beyond the relatively compact condition. This can formulated as
\begin{theorem}\label{Mharmonic}
Suppose that $n\geq 1$ and let $\Omega$  be an open connected subset of  the complex unit ball $ \BCn$ in $\C^n$, containing the origin $0$ such that $\lambda(\Omega) <\infty$, where
$ \lambda$ is the measure with density $\frac{1}{(1-|z|^2)^{n+1}}$ with respect to the Lebesgue measure.
Assume that for every compact set \(K\subset\BCn\),
\begin{eqnarray}\label{MhC1cond}
\sup_{z\in K}
\int_\Omega |\xi-z|^{1-2n}\,d\lambda(\xi)
<+\infty,
\end{eqnarray}
\begin{eqnarray}\label{MhC1bcond}
\lim_{r\to1^-} (1-r^2)^{n-2}
\sup_{z\in\Omega\cap\{|z|=r\}}
\int_\Omega |\xi-z|^{1-2n}\,d\nu(\xi)=0
\end{eqnarray}
and
\begin{eqnarray}\label{MTopcond}
\partial\Omega\cap\BCn =\partial\overline{\Omega}\cap\BCn.
\end{eqnarray}
Suppose that, for every $\lambda$-integrable  \(M\)-harmonic function \(f\) in a neighbourhood of
\(\overline{\Omega} \cap\BCn\), one has
\[
\frac1{\lambda(\Omega)}\int_\Omega f(x)\,d\lambda(x)=f(0).
\]
Then \(\Omega\) is a ball centered at \(0\). Equivalently, there
exists \(0<R<1\) such that
\[
\Omega=\{x\in\BCn:\ |x|<R\}.
\]
\end{theorem}

2) We finally point out that the reduction of our method to the Euclidian case give the results \cite{AG} and \cite{Ku}. Indeed, in this case under the only assumption  (\ref{finiteVol}), both assumptions 
(\ref{C1cond}) and 
(\ref{C1bcond}) hold.

{\it Acknowledgements}: The authors would like to thank Miroslav Engli\v s for valuable comments and remarks.

\end{document}